\documentclass{amsart}
\usepackage[left=2.5cm,right=2.5cm, top=2.5cm, bottom=2.4cm]{geometry}
\usepackage{amsmath,amsfonts,amssymb}
\usepackage{url, hyperref}
\usepackage{color,soul} 
\usepackage[normalem]{ulem} 
\usepackage{enumerate,graphics, graphicx}
\usepackage{subfig}

\newtheorem{theorem}{Theorem}[section]

\newtheorem{prop}[theorem]{Proposition}

\newtheorem{cor}[theorem]{Corollary}
\theoremstyle{definition}
\newtheorem{defn}[theorem]{Definition}
\newtheorem{eg}[theorem]{Example}
\newtheorem{problem}[theorem]{Problem}

\newtheorem{rmk}[theorem]{Remark}
\newtheorem{notation}[theorem]{Notation}

\newcommand{\B}{\mathcal B}

\newcommand{\disc}[1]{\textrm{disc} #1}

\begin{document}
\title[Toric algebra of hypergraphs]{Toric algebra of hypergraphs}

\author{Sonja Petrovi\'c}
\address{Sonja Petrovi\'c is with the Department of Statistics, The Pennsylvania State University, University Park PA 16802}
\email{\url{petrovic@psu.edu}}

\author{Despina Stasi}
\address{Despina Stasi was with the Department of Mathematics, Statistics and Computer Science, University of Illinois at Chicago, Chicago IL 60607; and is currently at the Department of Statistics, The Pennsylvania State University, University Park PA 16802}
\email{\url{despina@psu.edu}}

\date{\today}


\begin{abstract}
The edges of any hypergraph parametrize a monomial algebra called the edge subring of the hypergraph. We study presentation ideals of these edge subrings, and describe their generators in terms of balanced walks on hypergraphs. Our results generalize those for the defining ideals of edge subrings of graphs, which are well-known in the commutative algebra community, and popular in the algebraic statistics community.
One of the motivations for studying toric ideals of  hypergraphs comes from algebraic  statistics, where generators of the toric ideal give a basis for random walks on fibers of the statistical model specified by the hypergraph. Further, understanding the structure  of the generators gives insight into the model geometry.
\end{abstract}

\maketitle


\section{Introduction}

The rich history of the ideal theory of graphs, dating back to the seminal paper \cite{SVV94}, shows that fundamental properties of monomial algebras associated to graphs have interesting combinatorial interpretations. The connection inspired a body of work that joins combinatorial and algebraic methods with at least two algebraic objects in focus: monomial edge ideals of graphs, and toric ideals of edge subrings of graphs.

The study of quadratic squarefree monomial ideals corresponding to edges of a graph  has generated a lot of interest from the algebraic community. In particular, the comprehensive monograph \cite{Vill} contains numerous references on the subject.
More recent work includes, for example, \cite{CorsoNagel09}  who describe Betti numbers and primary decompositions for edge ideals of restricted families of graphs.
On the other hand, presentation ideals of edge subrings of graphs, which are toric rather then monomial, arise as kernels of monomial maps whose images correspond to the edges of a graph.  These toric ideals have been studied in \cite{Vill95}, \cite{OH-quad, OH-multipartite},  \cite{Vill}, \cite{RTT12}.

Several results for the monomial edge ideals have been extended to the case of hypergraphs (see, for example, \cite{HvT08},  \cite{Vill-clutters}), where a hypergraph is a generalization of a graph with edges containing more than two vertices. However, toric ideals of edge subrings remain completely unexplored beyond the graph case.
Nevertheless, they are of interest in applications, and as blow-up algebras, since the edge subring of a uniform hypergraph is isomorphic to the special fiber ring of its edge ideal.

In this work, we extend the theory for these toric ideals from graphs to (uniform) hypergraphs. We study the combinatorial signatures of presentation ideals of monomial subalgebras parametrized by squarefree monomials of degree greater than $2$. Their generators, Gr\"obner and Graver bases are given in terms of monomial walks on hypergraphs,  generalizing the constructions that are well-known for graphs.

One of the motivations for studying toric ideals of  hypergraphs comes from algebraic  statistics, where generating sets of toric ideals play an important role in testing how well a model fits the given data. The edge subring of a hypergraph corresponds to any exponential family model whose joint probabilities are parametrized by monomials. For more details, see Section~\ref{sec:relevance}.
In fact, edge subrings of graphs are already known in algebraic statistics, most notably in \cite{Morton10}, \cite{OHT:11}, \cite{StWe12}, \cite{PRF:09}.
The first author has used  these toric ideals to gain insight into the geometry of a particular random graph model, existence of maximum likelihood estimators, and a better way to generate Markov moves \cite{PRF:09, RPF:11}.
In particular, this work is a starting point for developing the theory required for the more general algebraic statistical models.

\bigskip
A hypergraph $H$ is \emph{$d$-uniform} if every edge contains $d$ vertices.
We describe the natural  one-to-one correspondence between sets of squarefree monomials of degree $d$ and $d$-uniform hypergraphs with the following notation.
For a finite $d$-uniform hypergraph $H$ on $V=\{x_1,\dots,x_n\}$ and edge set $E$, each edge $e_i$ of $H$ encodes a squarefree monomial
$
	x^{e_i}:=\prod_{j\in e_i} x_j
$
of degree $d$ in the polynomial ring $k[x_1,\dots,x_n]$ over a field $k$.  Thus, the hypergraph $H$ can be written as a set of monomials representing the edges:
$
 	H = \{x^{e_1}, \dots, x^{e_k}\}.
$
The \emph{edge subring} of the hypergraph $H$, denoted by $k[H]$, is the monomial subring of the polynomial ring generated by the edges of $H$; in symbols,
\[
	k[H] := k[x^{e_i} : e_i \in E(H)].
\]
Letting $t_{e_i}$ be a {variable} representing the edge $e_i$, define a ring homomorphism
$	\phi_H : k[t_{e_i}]  \to k[H]$ with
$ \phi_H(t_{e_i})=x^{e_i}$.
The kernel of this monomial map, denoted by $I_H$, is the \emph{toric ideal of the edge subring of the hypergraph $H$}. It encodes the algebraic relations among the edges of the hypergraph.

\bigskip

A first problem of interest is to describe the combinatorics of generators of the ideal for an arbitrary hypergraph. With this in mind, we define  \emph{monomial walks} on hypergraphs and the balancing condition in Section~\ref{sec:walks}, which are the natural direct generalization of monomial walks on graphs. 
 Theorem~\ref{thm:balancedEdgesEquivalence} characterizes the binomials in $I_H$ in terms of such monomial walks, extending the classical theorems of Villarreal \cite{Vill95,Vill}, and Ohsugi and Hibi \cite{OH-quad}. 
In Section~\ref{sec:relevance} we outline the relevance of our results: the correspondence between statistical models and hypergraphs is explored in Section~\ref{subsec:statistics}, and we derive a connection to the well-studied set-theoretical problem of combinatorial discrepancy in Section~\ref{subsec:discrepancy}. 
Given the apparent difficulty of the problem of describing \emph{all} primitive monomial walks on an arbitrary hypergraph, we dedicate Section~\ref{sec:sparseMonomialHypergraphs}  to sparsely intersecting hypergraphs, which generalize the three types of primitive walks on graphs. 
In particular, Propositions~\ref{prop:matchings}, \ref{prop:sunflowers}, \ref{prop:partitionedCoreSunflowers} and Theorem~\ref{thm:sunflowersRelaxedCore}  describe several supporting hypergraphs of primitive monomial walks. Their basic building blocks consist of  matchings and sunflowers, well-known in the hypergraph literature.  
In Section~\ref{sec:ctrexamples} we  show that many of the results for graphs cannot be generalized, and the expected degree bounds do not hold.
Finally, Section~\ref{sec:problems} outlines some of the many open problems that we hope will be addressed by the combinatorics community.

Algebraic properties of  toric ideals of hypergraphs such as normality, Cohen-Macaulayness, and geometry of the corresponding polytopes are all open problems of interest, but are beyond the scope of this paper. Degree bounds of generators are studied in  \cite{GrPe}. The problem of finding these generators algorithmically and understanding the complexity of such algorithms remains open and would make a significant contribution to applied algebraic statistics.


\section{Monomial Walks on Hypergraphs}
\label{sec:walks}

We begin with an example in the case when $H$ is $2$-uniform.
\begin{eg} \label{eg:k22}
	Consider the simplest case when $H$ is a graph, say, the complete graph $H=K_5$ on vertices $x_1,\dots,x_5$.
	The edge subring $k[K_5]\subseteq k[x_1,\dots,x_5]$ is parametrized by
	\begin{align*}
		\phi_{K_5} : k[t_{ij}]  &\to k[x_1x_2, x_1x_3, \dots, x_4x_5]\\
		t_{ij}	&\mapsto  x_ix_j.
	\end{align*}
	The toric ideal $I_{K_5}$ is generated by the following 10 quadrics:
	\begin{align*}
		t_{24}t_{35}-t_{23}t_{45}  , \quad
 t_{14}t_{35}-t_{13}t_{45}  ,   \quad
 t_{34}t_{25}-t_{23}t_{45}   ,  \quad
 t_{14}t_{25}-t_{12}t_{45}    ,     \quad
t_{13}t_{25}-t_{12}t_{35}      ,  \\
 t_{34}t_{15}-t_{13}t_{45}      ,    \quad
t_{24}t_{15}-t_{12}t_{45}        ,      \quad
 t_{23}t_{15}-t_{12}t_{35}        ,       \quad
 t_{13}t_{24}-t_{12}t_{34}         ,      \quad
 t_{23}t_{14}-t_{12}t_{34}  .  
	\end{align*}
	For example, the binomial $f:=t_{13}t_{24} - t_{12}t_{34}  \in I_H$ is represented in the leftmost graph of Figure~\ref{fig:3walksOnGraphs}.
	Since $f$ is represented by an even cycle with alternating colors on the edges, we say that it \emph{arises} from the cycle.

\end{eg}
In general, when $H$ is a graph, then the generators of the ideal $I_H$ are known in terms of these red-blue colorings on the edges of the graph:

\begin{theorem}[\cite{Vill95}, \cite{OH-quad}, see also \cite{Vill}] \label{thm:OH:walks}
	The toric ideal of the edge subring of a graph $G$ is generated by binomials arising from (primitive) even closed walks on $G$.
\end{theorem}

Such binomials have been characterized by the following result which, in fact, provides a Graver basis for the toric ideal.
\begin{theorem}[\cite{Vill95}, \cite{OH-quad}, see also \cite{Vill}]\label{thm:OH:characterization}
	Primitive even closed walks on $G$ are one of the following: (i) even cycles, (ii) two odd cycles sharing a vertex, or (iii) two odd cycles such that there are two walks connecting a vertex in one cycle with a vertex in the other.
\end{theorem}
\begin{figure}[htb]
\centering
\includegraphics[trim = 6.3cm 15.3cm 6cm 11.2cm, clip]{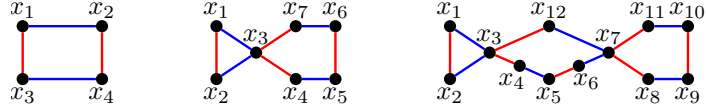}  
\caption{The three types of walks from Theorem~\ref{thm:OH:characterization}.  Colors are explained in Remark~\ref{rmk:colors}.} \label{fig:3walksOnGraphs}
\end{figure}

An even closed walk is the exact structure in a graph required to give rise to a binomial: the sequence of edges in the walk induces a bipartition of the edges with multiplicity and all vertices in the walk have the same degree in each of the two parts. We propose the following generalization of an even closed walk on a hypergraph. Following \cite{Vill}, we call it a monomial walk.

\begin{defn}
	A \emph{monomial walk} on a $d$-uniform hypergraph is an even sequence of edges $\mathcal W := \{e_1, e_2, \dots, e_{2k}\}$, where each edge $e_i$ in the sequence intersects the subsequence $\{e_1, e_2, \dots, e_{i-1}\}$, and each vertex $v\in\cup_{i=1}^{2k} e_i$ covered by $\mathcal W$ satisfies the following balancing condition:
$		| \{ e_i \in\mathcal W : v\in e_i \text{ and $i$ is odd}  \} |
		=
		| \{ e_i \in\mathcal W: v\in e_i \text{ and $i$ is even}  \} |
$.

	Each monomial walk gives rise to a binomial in the ideal $I_H$. Namely, we say that a binomial $f_{\mathcal W}$ \emph{arises from} $\mathcal W$ if
	\[
		f_{\mathcal W} := \prod_{i=1}^{k} t_{e_{2i-1}} - \prod_{i=1}^{k} t_{e_{2i}}.
	\]
	
	A monomial walk is said to be  \emph{primitive} if there does not exist a proper subwalk, that is, a collection $\mathcal W'$ of odd- and even-numbered edges, properly contained in $\mathcal W$, such that $\mathcal W'$ is also a monomial walk.
\end{defn}

For the remainder of the paper, we occasionally refer to a monomial walk as a walk. For convenience, all hypergraphs in this paper will be uniform. In the non-uniform case, binomials in $I_H$ are not necessarily supported by  alternating  monomial walks; although one can extend the balancing condition to the non-uniform case. 
Algebraically, we are simply considering homogeneous toric ideals. 

\begin{rmk}[bicoloring the edges in $f_{\mathcal W}$]\label{rmk:colors}
The classical theorems for graphs use odd and even numbered edges to describe $f_{\mathcal W}$.  For hypergraphs, however, we have found this notation and related pictorial representations extremely tedious. Thus, from now on, we refer to the even- and odd-numbered edges as blue and red edges. This bicoloring is implicit in the graphs case as well.
\end{rmk}

\begin{eg}
	Consider a complete $3$-uniform hypergraph on $12$ vertices. Figure~\ref{fig:3walks} shows three monomial walks on this hypergraph.
\begin{figure}[hbt]
	\centering
	 \subfloat[]{\label{fig:walk1}\includegraphics[scale=0.45]{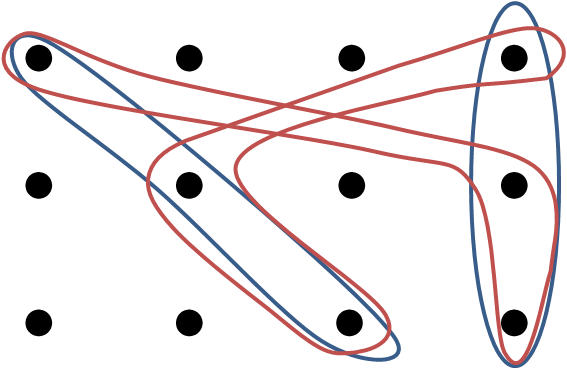}}
	\hspace{1cm}
	 \subfloat[]{\label{fig:walk2}\includegraphics[scale=0.45]{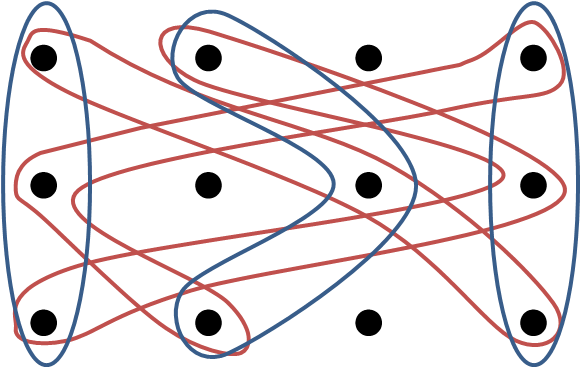}}
	\hspace{1cm}
	 \subfloat[]{\label{fig:walk3}\includegraphics[scale=0.45]{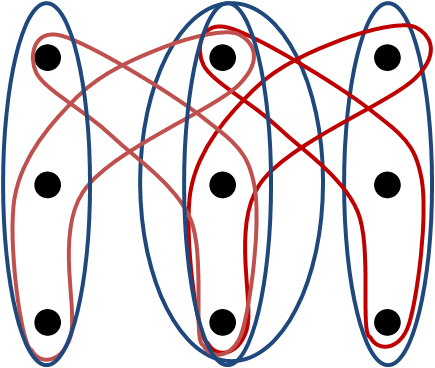}}
	\caption{Examples of Walks
}\label{fig:3walks}
	\end{figure}
\end{eg}

In general, 	there is no squarefree assumption on the binomial $f_{\mathcal W}$. Walking along the same edge multiple times during $\mathcal W$ is allowed, so it is possible that $e_i=e_j$ for $i\neq j$. (See Figure~\ref{fig:3walks}(C) for an example.)
If an edge in a monomial walk appears as both red and blue, then these two copies of the edge can be safely removed. The remaining set of edges constitutes another monomial walk. On the binomial side, this corresponds to a multiple of a smaller binomial.

\begin{defn}
\label{defn:balancedSetOfEdges}
Let $\mathcal E$ be a (multiset) collection of an even number of edges in $H$.  We say that $\mathcal E$ is \emph{balanced} with respect to a given bicoloring of $H$
if for each vertex $v$ covered by $\mathcal E$, the number of red edges containing $v$ equals the number of blue edges containing $v$; in symbols:
\[	\label{degreeCondition}\tag{$\star$}
	\deg_{blue} (v) = \deg_{red} (v).
\]
\end{defn}
\smallskip

The constructions above imply the following.
\begin{theorem}\label{thm:balancedEdgesEquivalence}
Let $H$ be a uniform hypergraph. 
	Any balanced collection of edges $\mathcal E\subset E(H)$ constitutes a monomial walk, or a collection of disjoint monomial walks.

    In particular, the ideal $I_H$ is generated by binomials $f_{\mathcal W}$ arising from primitive monomial walks $\mathcal W$ on ${H}$.
\end{theorem}
\begin{proof}
A binomial $f$ is in the ideal $I_H$ if and only if the set of edges $\mathcal E$ appearing in the support of the binomial is balanced. Namely, $f\in I_H$ is given by a collection of edges such that
\[
	f=\prod_{\text{blue edges }e} t_e - \prod_{\text{red edges }e'} t_{e'}
\]
where the coloring is induced by the binomial $f$ itself. Being in the kernel of $\varphi_{H}$ is equivalent to
\[ 	\varphi_{H}\left(\prod_{\text{blue edges }e} t_e\right)=\varphi_{H}\left(\prod_{\text{red edges }e'} t_{e'}\right).
\]
This, in turn, is equivalent to Condition~\eqref{degreeCondition}.

	We need to show that any balanced collection of edges with respect to some bicoloring can be written as arising from a monomial walk, or a set of disjoint monomial walks.  	
	Throughout the proof we will maintain a sequence of visited edges $\mathcal W$ from $\mathcal E$, and the set of red and blue degrees of each vertex $v$ covered by $\mathcal E$, which we denote by $\deg_{red}(v;\mathcal E)$ and $\deg_{blue}(v;\mathcal E)$ respectively.
	
	We begin with $\mathcal W=\emptyset$. Pick an arbitrary blue edge $e\in\mathcal E$, define $e_1:=e$, and update $\mathcal W=\mathcal W\cup \{e_1\}$, $\mathcal E=\mathcal E - \{e_1\}$, noting that $\deg_{blue}(v;\mathcal E)$ has decreased by $1$, for every $v\in e_1$. Pick an arbitrary $u_1\in e_1$. By~\eqref{degreeCondition}, there exists a red edge $e'\in\mathcal E$ such that $u\in e'$. Let $e_2:=e'$ and update $\mathcal W=\mathcal W\cup\{e_2\}$, $\mathcal E=\mathcal E - \{e_2\}$, noting that $\deg_{red}(v;\mathcal E)$ has decreased by $1$, for every $v\in e_2$. Clearly, $u_1$ is now balanced in $\mathcal W$.

Repeat recursively as follows. Check whether $\deg_{blue}(v;\mathcal E)=\deg_{red}(v;\mathcal E)$ for all $v$. If not, there must exist an unbalanced vertex covered by $\mathcal W$. For odd (resp.\ even) $i$, there exists a blue (resp.\ red) edge in $\mathcal E$ that intersects an even (resp. odd) edge $e_j$ in $\mathcal W$, $j<i$. Call this edge $e_i$ and move it to $\mathcal W$.
 
On the other hand, if $\deg_{blue}(v;\mathcal E)=\deg_{red}(v;\mathcal E)$ holds for every $v$, we have produced a monomial walk $\mathcal W$. In the case where $\mathcal E$ is not empty, the algorithm restarts on the current $\mathcal E$ with empty $\mathcal W$ to find additional edge-disjoint monomial walks. Finally, the algorithm stops when $\mathcal E$ is empty.
	
What we have shown is that every binomial in the ideal $I_H$ arises from a finite collection of monomial walks. That only primitive binomials are necessary to generate the ideal follows from standard arguments: e.g. see Lemma 4.6 in \cite{St}.
\end{proof}

Naturally, if a certain edge appeared multiple times in $\mathcal E$, each occurrence of the edge gets a unique label, so that the above argument is valid for multisets of edges.

Recall that the Graver basis is a special generating set of the toric ideal, usually very large and algebraically redundant, but combinatorially nice.
What we have shown above is that the Graver basis of $I_H$ consists of primitive monomial {walks} with support on the edge set of ${H}$.
In addition, as in the graphs case, since the Graver basis contains all reduced Gr\"obner bases, we can say the following:
\begin{cor}\label{cor:grobner-is-a-walk}
	If $f$ is a polynomial in any reduced Gr\"obner basis of $I_H$, then $f=f_{\mathcal W}$ for some primitive monomial walk $\mathcal W$.
\end{cor}

The natural problem now is to understand primitive monomial walks on a given hypergraph $H$. We will concern ourselves with a criterion which will allow us to determine if a given set of edges is a primitive monomial walk.  To this end, we make the following definition.
		
\begin{defn}\label{defn:monomialH}
	A multihypergraph $\mathcal H$ is called a \emph{monomial hypergraph} if its edges can be colored so that they are balanced (see Definition~\ref{defn:balancedSetOfEdges}). Such an $\mathcal H$ is \emph{primitive} if there exists a balanced bicoloring that makes the corresponding monomial walk primitive.
\end{defn}

The purpose of this definition is twofold.
First, note that a monomial hypergraph corresponds to a monomial walk.
However, a given walk comes equipped with the blue and red bicoloring of the edges (i.e., the binomial is given explicitly); whereas in a monomial hypergraph, the coloring is not fixed (i.e., it represents the support of the binomial).
In the graph case, this distinction is less important because the coloring is more obvious. For example, when we say that ``even cycles are generators of the toric ideal of a graph'', we do not specify how to color the edges of the cycle; however, upon the inspection of the proof, of course, it is obvious that the edges are colored in alternating colors to give rise to a binomial. In hypergraphs, such a choice of coloring may not always be immediate or unique, so this distinction is important.
Second, a monomial hypergraph is not a simple hypergraph: there is no squarefree assumption on the monomial walks, so an edge in a hypergraph may be traced multiple times in a walk over its edge set. This generality is required since we cannot assume that the binomials in $I_H$ are squarefree.

\begin{cor}\label{cor:ugb}
	The set of binomials
	\[
		\{f_{\mathcal W}  | \mathcal W \text{ is a bicolored primitive  monomial hypergraph over the hypergraph } H \}
	\]
	is a universal Gr\"obner basis for the toric ideal $I_H$.
\end{cor}

In summary: a monomial walk has an alternating coloring on the edges and requires that all the visited vertices are balanced with respect to that coloring; and a monomial hypergraph is a multihypergraph with the supporting edge set being that of $H$, it corresponds to a monomial walk, but its coloring is not fixed.
\emph{The monomial hypergraph corresponding to the monomial walk $\mathcal W$ is the combinatorial signature of the binomial $f_\mathcal W\in I_H$.}

\section{Monomial hypergraphs: from log-linear models to discrepancy} \label{sec:relevance}
 
Before exploring the structure of monomial hypergraphs in the next section, let us briefly reflect on their relevance. 
We explain how they arise naturally in algebraic statistics, and proceed to show how they relate to combinatorial discrepancy of hypergraphs.

\subsection{Parameter hypergraphs}\label{subsec:statistics}

Monomial maps make a notable appearance in statistics: they encode important families of statistical models for discrete random variables called \emph{log-linear models}. A log-linear model is a family of joint probability distributions that are described by monomials in the model parameters (equivalently, minimal sufficient statistics of the model are in the row span of a linear map, and hence the name; but let us not dwell on this here). To see how monomial maps arise, let us consider the following very simple example. Two random variables $Z_1$ and $Z_2$, taking $a$ and $b$ values, respectively, are said to be independent if the joint probability $P_{ij}:=Prob(Z_1=i,Z_2=j)$ factors as follows: $P_{ij} = Prob(Z_1=i)Prob(Z_2=j)$ for all $i\in[a]$ and $j\in[b]$. Denote the probabilities $Prob(Z_1=i)$ and $Prob(Z_2=j)$ by $x_i$ and $x_{a+j}$, respectively. The unknown quantities $x_1,\dots,x_{a+b}$ are  called the \emph{model parameters}; they completely determine this family of distributions. We can summarize this simple example by saying that the {independence model} for $Z_1$ and $Z_2$ is specified by the monomial map that sends $P_{ij}$ to $x_ix_{a+j}$.
In this sense, the edge subring $k[K_{a,b}]$ of the complete bipartite graph $K_{a,b}$ on the vertex set $\{x_1,\dots,x_a\}\sqcup \{x_{a+1},\dots,x_{a+b}\}$ describes the structure of the joint probability distributions that belong to the independence model.

This correspondence can be  naturally extended to any more complicated model for discrete random variables $Z_1,\dots,Z_m$, whose joint probabilities are parametrized by squarefree monomials in the model parameters $x_1,\dots,x_n$. 
Each such log-linear model $\mathcal M$ is encoded by a hypergraph $H_{\mathcal M}$ on the vertex set $x_1,\dots,x_n$, constructed as follows:  
$\{x_j\}_{j\in J}$ is an edge of $H_\mathcal M$ if and only if the index set $J$ describes one of the joint probabilities in the model; that is, there exist values $i_1,\dots,i_m$ such that  $Prob(Z_1=i_1,\dots,Z_m=i_m)=\prod_{j\in J} x_j$. 

Henceforth, we will name $H_\mathcal M$ the \emph{parameter hypergraph of the model $\mathcal M$}. 
Somewhat surprisingly, toric ideals $I_{H_\mathcal M}$ are quite relevant for testing goodness of fit of the model $\mathcal M$. Without going into statistical details, a non-asymptotic test for model/data fit is carried out by exploring the fiber of a given point in the image of the monomial map $\phi_{H_\mathcal M}$. 
In terms of the hypergraph $H_\mathcal M$, an observed data point corresponds to a collection of edges; for convenience, think of them as colored blue. The fiber of that data point then consists of all red edge sets that balance it, and thus  have the same image under the map $\phi_{H_\mathcal M}$.   
The recipe for exploring a fiber is to move from any point to another point by applying a \emph{Markov move}, which corresponds to  a generator of the toric ideal $I_{H_\mathcal M}$ of the model.  Since the binomials in $I_{H_\mathcal M}$ are supported by balanced edge sets, this move is interpreted as  removing the observed (blue) edges and replacing them by the red edges, arriving at another monomial with the same image. 
 The correspondence between Markov moves and toric generators is well-known in algebraic statistics literature, and often referred to as the fundamental theorem of Markov bases.  
The technical description of how such a random walk  on a fiber is actually implemented is beyond the scope of this paper; but the interested reader is referred to  \cite{algStatBook} and the recent text \cite{AHT2012}.

In general, finding Markov bases for many relevant statistical modes is a wide-open problem; some complexity bounds and a few structural theorems are known.  
Therefore, our goal is to understand the combinatorial structure of monomial hypergraphs supported on $H_\mathcal M$. Instead of focusing on one log-linear model, we explore combinatorial signatures of general monomial hypergraphs.  We focus on homogeneous ideals, thus the models parametrized by uniform hypergraphs.
Finally, as explained in \cite{AHT2012}, for homogenous toric models with statistical sampling constraints or with structural zeros in the model, a minimal generating set of the toric ideal may not suffice for connecting the fibers, and instead squarefree part of the Graver basis is used. 

As mentioned in the introduction, the edge subring of a \emph{graph} has appeared in several results in algebraic statistics. By defining the parameter hypergraph $H_\mathcal M$, which the reader should agree is quite a natural generalization, we give an applied motivation to extend the well-known results on toric ideals of graphs to the more general case. By Theorem~\ref{thm:balancedEdgesEquivalence}, Markov moves for any (homogeneous) log-linear model $\mathcal M$ are encoded by primitive monomial hypergraphs supported on a (uniform) $H_{\mathcal M}$.  
One of our main goals is to frame this problem in terms of combinatorics of hypergraphs.  

\subsection{A dual problem: combinatorial discrepancy}\label{subsec:discrepancy}

To that end, we begin by noting a relation between monomial hypergraphs and zero-discrepancy hypergraphs. Broadly, the combinatorial discrepancy problem asks for a coloring of the vertices of a hypergraph \mbox{$H=(V, E)$} such that each hyperedge contains roughly the same number of vertices of each color. (The reader is directed to \cite{Matousek2010, Chazelle2002} for a detailed introduction to discrepancy theory.) Let us formally define the discrepancy of a hypergraph.

\begin{defn}
The \emph{discrepancy} of a hypergraph $H=(V,E)$, denoted  $\disc(H)$, is the quantity $\disc(H)=\min_{\chi: V\to \pm1} \max_{e\in E} \sum_{v\in e} \chi(v)$. In words, it is the minimum, over all colorings $\chi$, of the maximum discrepancy of an edge in $H$.
\end{defn}

Given a hypergraph $H$, the \emph{dual hypergraph} $H^*$ is the hypergraph containing a vertex $v^*$ for each edge $e \in E(H)$ and an edge $e^*$ for each $v\in V(H)$; here $e^*=\{v^*\in V(H^*): v^* \text{ corresponds to an edge }e\in E(H) \text{ that is incident in $H$ to vertex }v\}$. More briefly, the vertex-edge incidence matrix of $H^*$ is the transpose of the vertex-edge incident matrix of $H$. 

It is not hard to see that a zero-discrepancy hypergraph $H^*=(V^*,E^*)$ is dual to a monomial hypergraph $H=(V,E)$: the mapping $\chi: V^*\to \pm 1$ acts on the vertices of $H^*$ but it acts on the edges of $H$. Zero discrepancy for $H^*$ implies that each edge of $H^*$ contains the same number of red and blue vertices, and by consequence each vertex of $H$ is incident to the same number of red and blue edges.  This establishes the following theorem.

\begin{theorem}\label{thm:discrepancy}
A hypergraph $H$ is a monomial hypergraph if and only if $\disc(H^*)=0$. Furthermore, a monomial hypergraph $H$ is primitive if and only if no coloring $\chi: V^*\to \pm 1$ realizing $\disc(H^*)=0$, also induces a zero discrepancy coloring on a (non-empty) subhypergraph $\tilde{H}$ of $H^*$ produced by deleting any subset of vertices and shrinking the edges.  
\end{theorem}

Deleting a vertex from $H^*$ and shrinking its incident edges is equivalent to deleting an edge in the dual hypergraph $H$. There are some other interesting facts here. While a monomial hypergraph can have multiple edges, this is not true or necessary for the dual zero-discrepancy hypergraph, as multiple edges in $\mathcal{H}$ would correspond to vertices that are incident to the same set of edges in the dual. 

An important inapproximabillity result of \cite{CharikarNeNi2011} on discrepancy implies that it is NP-Hard to decide whether a hypergraph \mbox{$H=(V,E)$} is the support for a monomial hypergraph, or whether for every bicoloring of $E(H)$, there exists a vertex $v\in V(H)$ such that \mbox{$\left|\deg_{red}(v)-\deg_{blue}(v)\right|\geq \Omega(\sqrt{|E(H)|)}$}.

Given the difficulty of deciding the discrepancy of a hypergraph, we pursue a structural theorem for monomial hypergraphs in the next section.


\section{Sparse Bouquets}  \label{sec:sparseMonomialHypergraphs}

Theorem~\ref{thm:balancedEdgesEquivalence} can be used to give a characterization of (primitive) monomial walks in terms of the existence of an edge partition satisfying the degree condition~(\ref{degreeCondition}). However it gives no other information about the structure of the supporting hypergraphs. 

In this section, we use Theorem~\ref{thm:balancedEdgesEquivalence} to construct and study supports of monomial hypergraphs that parallel the graphs case. For example, consider case $(ii)$ from Theorem~\ref{thm:OH:characterization}. Two odd cycles glued at a vertex support a primitive monomial walk on a graph.  Increasing the degree of that core vertex could produce monomial walks, but they can never be primitive. In contrast, the core vertices in hypergraph supports of primitive walks can have degree larger than $2$; in fact, it can be arbitrarily high!

The first three results in this section directly generalize cases $(i)$, $(ii)$ and $(iii)$ (for short cycles) from Theorem~\ref{thm:OH:characterization}. 
First, a pair of perfect matchings generalizes an even cycle, and its supporting hypergraph is characterized in Proposition~\ref{prop:matchings}.
Second, a monomial sunflower generalizes a bow-tie (two $3$-cycles sharing a vertex),  and its supporting hypergraphs are characterized in Proposition~\ref{prop:sunflowers}.
Finally, a partitioned core sunflower generalizes two 3-cycles connected by two paths, and the supporting hypergraphs are described in Proposition~\ref{prop:partitionedCoreSunflowers}.

A further generalization summarizing the above propositions is given in Theorem~\ref{thm:sunflowersRelaxedCore}. So, unlike in the case of graphs, these three results do not exhaustively characterize monomial hypergraphs. Nevertheless, starting from sunflowers and matchings, natural building blocks of monomial hypergraphs, we can obtain much more complex structures. We delve more into the open problem of structurally characterizing primitive monomial hypergraphs in Section~\ref{sec:problems}.

\subsection{Matchings}

\begin{defn}A \emph{matching} on a hypergraph $H=(V,E)$ is a subset $M\subset E$ of independent edges, that is, no two edges intersect. A matching is called \emph{perfect} if it covers all the vertices of the hypergraph, i.e., $V(M)=V$. \end{defn}

The simplest monomial hypergraph can be formed using a pair of perfect matchings on the same set of variables. Next we give a sufficient and necessary condition for the primitivity of such a monomial hypergraph. A hypergraph $H=(V,E)$ is said to be \emph{connected} if its primal graph is connected, where the primal graph has the same vertex set and an edge between any two vertices contained in the same hyperedge.

\begin{prop}[Pair of matchings]\label{prop:matchings}
    Let $H=(V,E)$ be a $d$-uniform hypergraph such that $E=M_r\sqcup M_b$, where $M_r$, $M_b$ form two edge-disjoint perfect matchings on the vertex set $V$. 
    A monomial hypergraph $\mathcal H$ with support $H$ is primitive if and only if $H$ is connected and $\mathcal H$ contains no multiple edges.
\end{prop}

\begin{proof}
     Let $\mathcal H$ be a primitive monomial hypergraph and consider a primitive monomial walk $\mathcal W$ on the edges of $\mathcal H$ with an appropriate associated bicoloring. Any walk is connected, so suppose for contradiction that edge $e\in E$ has multiplicity $k$ in $\mathcal W$. By construction of the supporting hypergraph $H$, there exist exactly two distinct edges in $H$, say $e$ and $f$, containing any vertex $v$. We may assume that all copies of $e$ in $\mathcal W$ are red, as two copies of $e$ in $\mathcal W$ that belong to a different color partition would create a trivial subwalk. By the degree condition~(\ref{degreeCondition}) there must be $k$ copies of $f$ in $\mathcal W$, all colored blue. Edges $e$ and $f$ are distinct so there exists a vertex $u$ in $f$ and a unique edge $g$ in $H$ such that $g$ contains $u$ and has multiplicity $k$ in $\mathcal W$. But $H$ is connected, so we may repeat this process for each vertex in $\mathcal W$ to conclude that all edges in $E$ appear in $\mathcal W$ in $k$ copies of the same color, contradicting the primitivity of $\mathcal W$.

    For sufficiency, we need only produce a bicoloring for the edges of $H$, taken with no multiplicity, such that the associated walk is primitive. 
    We construct the required bicoloring by coloring edges in $M_r$ red and edges in $M_b$ blue. Assume by way of contradiction that there exists a smaller walk $\mathcal W$ in $\mathcal H$ that is also a monomial walk. Let $V_{\mathcal W}$ and $E_{\mathcal W}$ be the vertex and edge sets induced by the walk $\mathcal W$, and let $\mathcal H_{\mathcal W}=(V_{\mathcal W},E_{\mathcal W})$. It is clear that every vertex in any  monomial walk, and so specifically in $\mathcal H_{\mathcal W}$ must have degree at least 2 (the sum of the blue degree and the red degree). Since $\mathcal W$ is a subwalk, there must be some edge or vertex of $\mathcal H$ that $\mathcal W$ does not contain. But every vertex in $H_{\mathcal W}$ has degree 2 so there can be no more edges in $\mathcal H\setminus \mathcal H_{\mathcal W}$ containing any vertex from $\mathcal H_{\mathcal W}$. Then $H$ is disconnected.
\end{proof}

Figures~\ref{fig:3walks}(A) and~\ref{fig:3walks}(B) are primitive matchings. The former is also commonly called a \emph{tight cycle}.

\subsection{Monomial Sunflower}

As pairs of matchings have the property that each vertex has degree $2$, they generalize even cycles on graphs. In contrast, the next case in Theorem~\ref{thm:OH:characterization} has a distinguished vertex of higher degree.  Its natural generalization in hypergraphs, which we will call a monomial sunflower, is the topic of this subsection, and is based on a very useful, highly structured hypergraph, the \emph{sunflower}, which, incidentally, is guaranteed to occur in hypergraphs with large enough edge sets, independently of the size of the vertex set. (See e.g.\ \cite{Jukna}.) 
 \begin{defn} \label{defn:sunflower}
A $d$-uniform hypergraph $H=(V,E)$ is a \emph{sunflower} if $e_i \cap e_j = C$ for all edges $e_i\neq e_j\in E$, and $C\subset V$. The set of vertices $C$ is called the \emph{core} of the sunflower, and each $e_i$ is called a \emph{petal}.
\end{defn}

In general, one may allow the core to be empty, in which case you simply get a matching. In this paper, we will explicitly use the name matching instead of allowing for empty cores.

For the remainder of the paper, fix the following notation.
\begin{notation}\label{notation:removingCores}
For a subset $C\subset V(H)$, define $H-C$ to be the hypergraph with vertex set $(V\smallsetminus C)$ and edges $\{e_i \smallsetminus C: e_i\subset E(H)\}$.
Similarly, for a subset $C\subset V(\mathcal H)$, define $\mathcal H-C$ to be the multihypergraph with vertex set $(V\smallsetminus C)$ and edges $\{e_i \smallsetminus C: e_i\subset E(\mathcal H)\}$.
The hypergraphs $H-C$ and $\mathcal H-C$ are not necessarily uniform and may include empty edges, or multiple copies of an edge. See Figures~\ref{fig:nonPrimitiveMonomialSunflower}, \ref{fig:primitiveMonomialSunflower} and \ref{fig:simpleNonPrimitiveBouquet} for examples.
\end{notation}

Clearly, there is no nontrivial monomial walk on a sunflower: since the degree of any non-core vertex is $1$, every petal needs to be both blue and red, resulting in the zero binomial.
Thus, a sunflower alone cannot be a support of a monomial hypergraph. The non-core petal vertices could, however, be balanced by a perfect matching. To that end, consider $H$ to be a sunflower with a perfect matching on its non-core vertices. In Proposition~\ref{prop:sunflowers}, we show that a primitive  monomial hypergraph with support $H$ can be characterized by counting petals.

Consider the set of connected components
\[
	 H-C = \bigcup_{j\in\mathcal{I}}G_{j}.
\]
For convenience, define $G_{\mathcal J}:=\bigcup_{j\in\mathcal{J}}G_{j},$ for any index set $\mathcal J\subseteq \mathcal I$. By abuse of notation we will write $\left|G_{j}\right|$ for the number of edges in the multiset $\{e\in H: e\cap C \neq\emptyset \text{ and } e\smallsetminus C\in G_{j}\}$.

The next proposition is based on the idea that counting petals detects balanced walks on $H$.
\begin{prop}[Sunflower]\label{prop:sunflowers}
Let $H$ be a hypergraph consisting of a sunflower and a perfect matching on the non-core vertices of the sunflower. The hypergraph $H$ is the support of a monomial hypergraph $\mathcal H$ if and only if there exists a partition $(G_\mathcal J, G_\mathcal K)$ of the connected components of $H-C$ such that \[\sum_{j\in \mathcal J} m_j|G_{j}|=\sum_{k\in \mathcal K} m_k|G_{k}|, \quad \text{where $m_j$, $m_k$ are integers}.\]

In addition, $\mathcal H$ is primitive if and only if the partition $(G_\mathcal J, G_\mathcal K)$ also satisfies the following two conditions:
\begin{enumerate}[i.]
\item no two subsets $\mathcal J'\subsetneq \mathcal J$ and $\mathcal K'\subsetneq \mathcal K$ satisfy $\sum_{j\in \mathcal J'} m'_j|G_{j}|=\sum_{k\in \mathcal K'} m'_k|G_{k}|$ for some integers $m'_j\leq m_j$, $m'_k\leq m_k$ and
\item there are no integers $m'_j, m'_k\in\mathbb Z$ such that the partition $(G_\mathcal J, G_\mathcal K)$ of the connected components of $H-C$ satisfies \mbox{$\sum_{j\in \mathcal J} m'_j|G_{j}|=\sum_{k\in \mathcal K} m'_k|G_{k}|$},  where $m'_j\leq m_j$, $m'_k\leq m_k$ and for at least one integer the inequality is strict. 
\end{enumerate}
\end{prop}
\begin{proof}
For the easy direction of the first statement, suppose $(G_\mathcal J, G_\mathcal K)$ is a partition of the connected components of $H-C$ such that $\sum_{j\in \mathcal J} m_j|G_{j}|=\sum_{k\in \mathcal K} m_k|G_{k}|$ for some $m_j,m_k\in\mathbb Z$. Construct the monomial walk $\mathcal W$ as follows: for every edge $e$ having a nonempty intersection with $G_{j}, j\in \mathcal J$ (respectively $G_{k}, k\in\mathcal K$), include in $\mathcal W$, $m_j$ (respectively $m_k$) copies of edge $e$. Color the matching edges red for $\mathcal J$ and blue for $\mathcal K$, and use the opposite colors for the petal edges. This immediately ensures that the non-core vertices satisfy the degree condition~(\ref{degreeCondition}). The degrees of a core vertex $v$ are $\deg_{blue}(v)=\sum_{j\in \mathcal J} m_j|G_{j}|$ and $\deg_{red}=\sum_{k\in \mathcal K} m_k|G_{k}|$ respectively.

For the other direction of the first statement, suppose $\mathcal H$ is a monomial hypergraph over $H$ and consider the monomial walk $\mathcal W$ consisting of the edges of $\mathcal H$. For every vertex $v$ in a component $G_i$ of $H$, there are exactly two distinct edges containing $v$: a petal edge $e$ and a matching edge $f$. Suppose $e$ appears with multiplicity $m_i$ in $\mathcal W$. Without loss of generality, we may assume that all $m_i$ copies of $e$ must belong to the same color partition, say red. By the degree condition (\ref{degreeCondition}), edge $f$ must also appear in $\mathcal W$ with multiplicity $m_i$, and all $m_i$ copies must belong to the blue color partition. Now there exists a vertex $u\in f\cap e'$ ($u\neq v$) where $e'$ ($e'\neq e$) is the unique petal edge containing $u$. To ensure $u$ is balanced, by the degree condition, $e'$ must also appear in $\mathcal W$ with multiplicity $m_i$. Repeating this argument over all vertices in the connected component $G_i$ we conclude that each matching edge in $G_i$ and each petal edge whose non-core vertices fall in $G_i$, must appear in $\mathcal W$ with multiplicity $m_i$. Then we may construct the partition $(G_\mathcal J, G_\mathcal K)$ so that $G_i\in G_{\mathcal J}$ if the matching edges it contains are colored red, and $G_i\in G_{\mathcal K}$ otherwise. Since the core vertices are also balanced, the equation $\sum_{j\in \mathcal J} m_j|G_{j}|=\sum_{k\in \mathcal K} m_k|G_{k}|$ holds, where the sum on the left hand side counts the blue degree of a core vertex and the right hand side its red degree.

The second statement in the proposition (primitivity) follows from the first statement and the definition of primitivity. Condition $i.$\ describes all subwalks not using all edges of the supporting hypergraph $H$ and condition $ii.$\ describes all subwalks using all edges of $H$ with smaller multiplicity then in $\mathcal H$.
\end{proof}

We may now make the following definition.
\begin{defn}
	A \emph{monomial sunflower} is a monomial hypergraph with support set consisting of a sunflower and a perfect matching on the non-core vertices.
\end{defn}

The partition $(G_{\mathcal J},G_{\mathcal K})$ corresponds to a bicoloring of the petals of the sunflower, as illustrated in the following example.

\begin{eg}
Figure~\ref{fig:nonPrimitiveMonomialSunflower} demonstrates that $H_1$ is the support of a simple non-primitive monomial sunflower and its components. Each edge appears exactly once in $\mathcal H_1$, and it is easy to see that for any two components in $H-C$, there is a bicoloring of $\mathcal H_1$ such that the two components form a subwalk.
\begin{figure}[hbt]
	\centering
	 \includegraphics[scale=0.55]{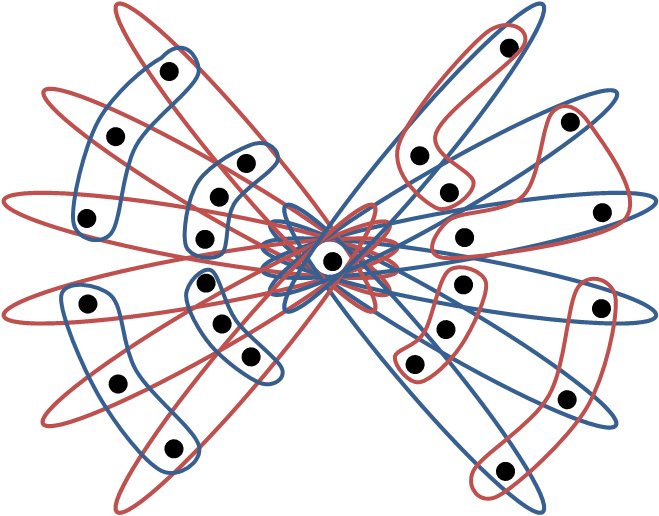}
    \hspace{1cm}
    \includegraphics[scale=0.55]{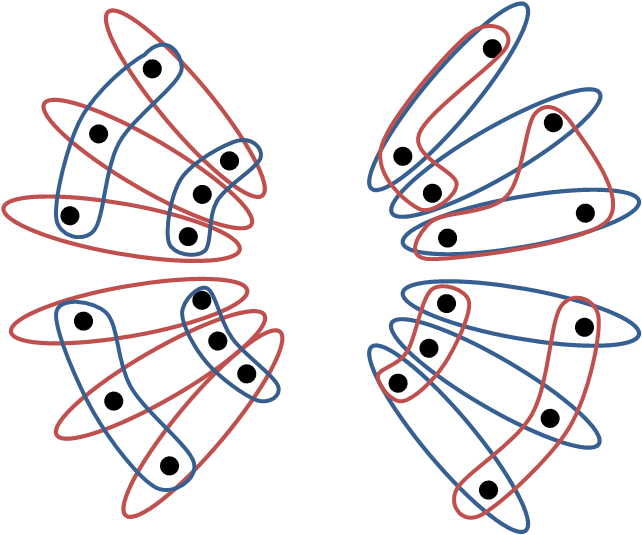}
	\caption{A walk on the non-primitive monomial sunflower $\mathcal H_1$ (left) and the components of $H_1-C$ with appropriate coloring and multiplicity (right).}
    \label{fig:nonPrimitiveMonomialSunflower}
\end{figure}

Figure~\ref{fig:primitiveMonomialSunflower} demonstrates a bicoloring of the monomial sunflower $\mathcal H_2$ that makes it primitive. The appropriate multiplicities for each component are $m_1=1$ and $m_2=2$. 
\begin{figure}[hbt]
	\centering
	 \includegraphics[scale=0.55]{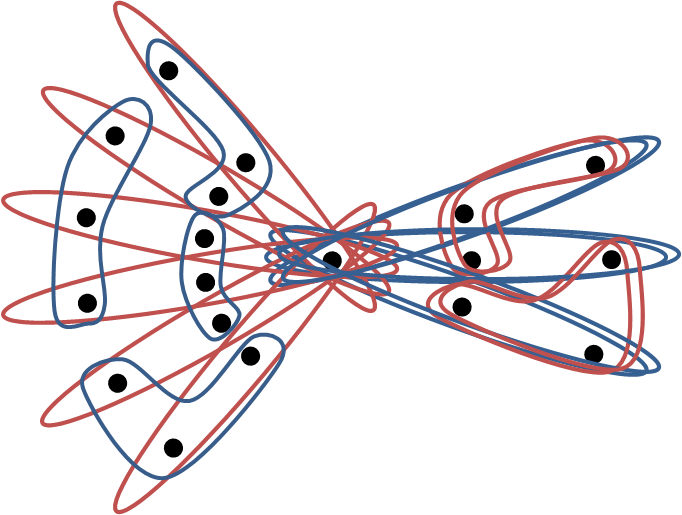}
    \hspace{1cm}
    \includegraphics[scale=0.55]{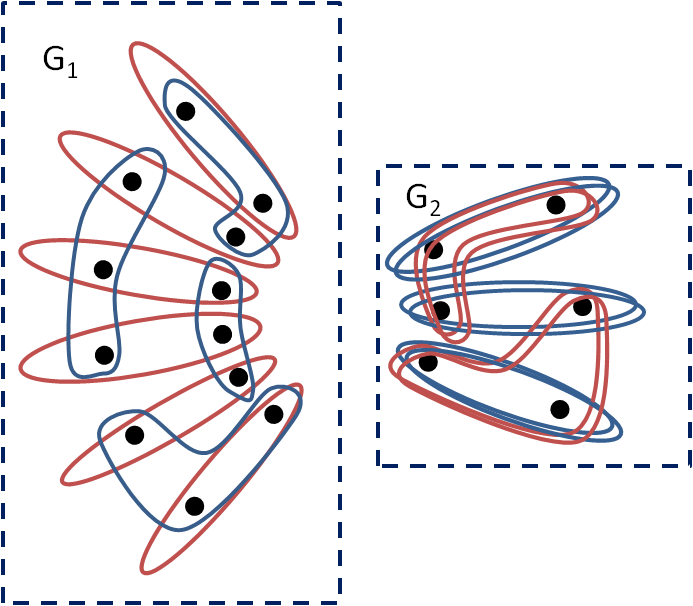}
	\caption{A walk on the primitive monomial sunflower $\mathcal H_2$ (left) and the components of $H_1-C$ with appropriate bicoloring and multiplicity (right).}
    \label{fig:primitiveMonomialSunflower}
\end{figure}
\end{eg}

\subsection{Independent Petal Bouquets}

To produce more general monomial hypergraphs and characterize primitivity on them, we may generalize the monomial sunflower by considering a hypergraph $H$ formed by a collection of sunflowers and a perfect matching on the subset of vertices that are not in the core of any sunflower. Let us, first, assume that the sunflowers are entirely disjoint.

Let $H$ be a hypergraph consisting of a collection of vertex-disjoint sunflowers $S_1, S_2, \ldots, S_{\ell}$ and a perfect matching on the non-core vertices of $\bigcup_i S_i$. We will call such an $H$ a \emph{matched-petal partitioned-core sunflower}. 

As we saw in Proposition~\ref{prop:sunflowers}, when $H$ contains a single sunflower, the components of $H-C$ and their multiplicities determine the degrees of the vertices of $H$ and ensure that the core is balanced. 
In a matched-petal partitioned-core sunflower, not all the components contribute to every sunflower $S_i$. So, one needs to keep track of the components of $H-C$ that contain vertices of $S_i$, for \emph{each} $S_i$. We employ the following notation.

\begin{notation}\label{notation:sunflowersSubset}
Let $H$ be a matched-petal partitioned-core sunflower. Given a fixed partition $(G_\mathcal J, G_\mathcal K)$ of the connected components of $H-C$, and a fixed sunflower $S_i$ of $H$, define $\mathcal{J}(S_i)\subset \mathcal J$ (respectively $\mathcal K(S_i)\subset \mathcal K$) to include every $j\in \mathcal J$ (respectively $j\in \mathcal K$), such that $G_j$ contains a vertex of $S_i$. In symbols, $\mathcal{J}(S_i)$ is the set of indices defined as $\mathcal{J}(S_i):=\{j\in\mathcal{J}(S_i): j\in \mathcal J \text{ and } e\cap G_j\neq \emptyset \text{ for some edge } e\in S_i\}$.
\end{notation}

\begin{prop}[Partitioned Core Sunflower]  \label{prop:partitionedCoreSunflowers}
Let $H$ be a matched-petal partitioned-core sunflower, and adopt notation~\ref{notation:sunflowersSubset}.

The hypergraph $H$ is the support of a monomial hypergraph $\mathcal H$ if and only if there exists a partition $(G_\mathcal J, G_\mathcal K)$ of the connected components of $H-C$ such that for each $1\leq i\leq\ell$:
\[\sum_{j\in \mathcal{J}(S_i)} m_j|G_{j}|=\sum_{k\in \mathcal{K}(S_i)} m_k|G_{k}|, \quad \text{where $m_j$, $m_k$ are integers.}\]

In addition, $\mathcal H$ is primitive if and only if the partition $(G_\mathcal J, G_\mathcal K)$ also satisfies the following two conditions:
\begin{enumerate}[i.]
\item no two subsets $\mathcal J'\subsetneq \mathcal J$ and $\mathcal K'\subsetneq \mathcal K$ satisfy for each $1\leq i\leq \ell$,
\[\sum_{j\in \mathcal{J}(S_i)\cap \mathcal J'} m'_j|G_{j}|=\sum_{k\in \mathcal{K}(S_i)\cap \mathcal K'} m'_k|G_{k}|,\]
for some integers $m'_j\leq m_j$, $m'_k\leq m_k$, and
\item there are no integers $m'_j, m'_k\in\mathbb Z$ such that the partition $(G_\mathcal J, G_\mathcal K)$ of the connected components of $H-C$ satisfies for each $1\leq i\leq \ell$,
    \[\sum_{j\in \mathcal{J}(S_i)} m'_j|G_{j}|=\sum_{k\in \mathcal{K}(S_i)} m'_k|G_{k}|,\]
    where $m'_j\leq m_j$, $m'_k\leq m_k$ and for at least one integer the inequality is strict.
\end{enumerate}
\end{prop}
\begin{proof}
The first statement follows with the same argument as Proposition~\ref{prop:sunflowers} for the non-core vertices. For a core vertex, given a partition $(G_\mathcal J, G_\mathcal K)$ of the connected components of $H-C$, the sum $\sum_{j\in \mathcal{J}(S_i)} m_j|G_{j}|$ (respectively $\sum_{k\in \mathcal{K}(S_i)} m_k|G_{k}|$) counts the number of petals of sunflower $S_i$ whose remnants fall in components $G_{\mathcal{J}(S_i)}$ (respectively $G_{\mathcal{K}(S_i)}$), and as such completely determines the quantity that each part of the bipartition contributes to the degree of a core vertex $v\in S_i$. As before, the second statement describes two types of subwalks.
\end{proof}

\begin{eg}
Figure~\ref{fig:mulipleCores} demonstrates a primitive partitioned core monomial sunflower $\mathcal H$. Note that all of  the multiplicities are equal to one.
\begin{figure}[hbt]
	\centering
	\includegraphics[scale=0.55]{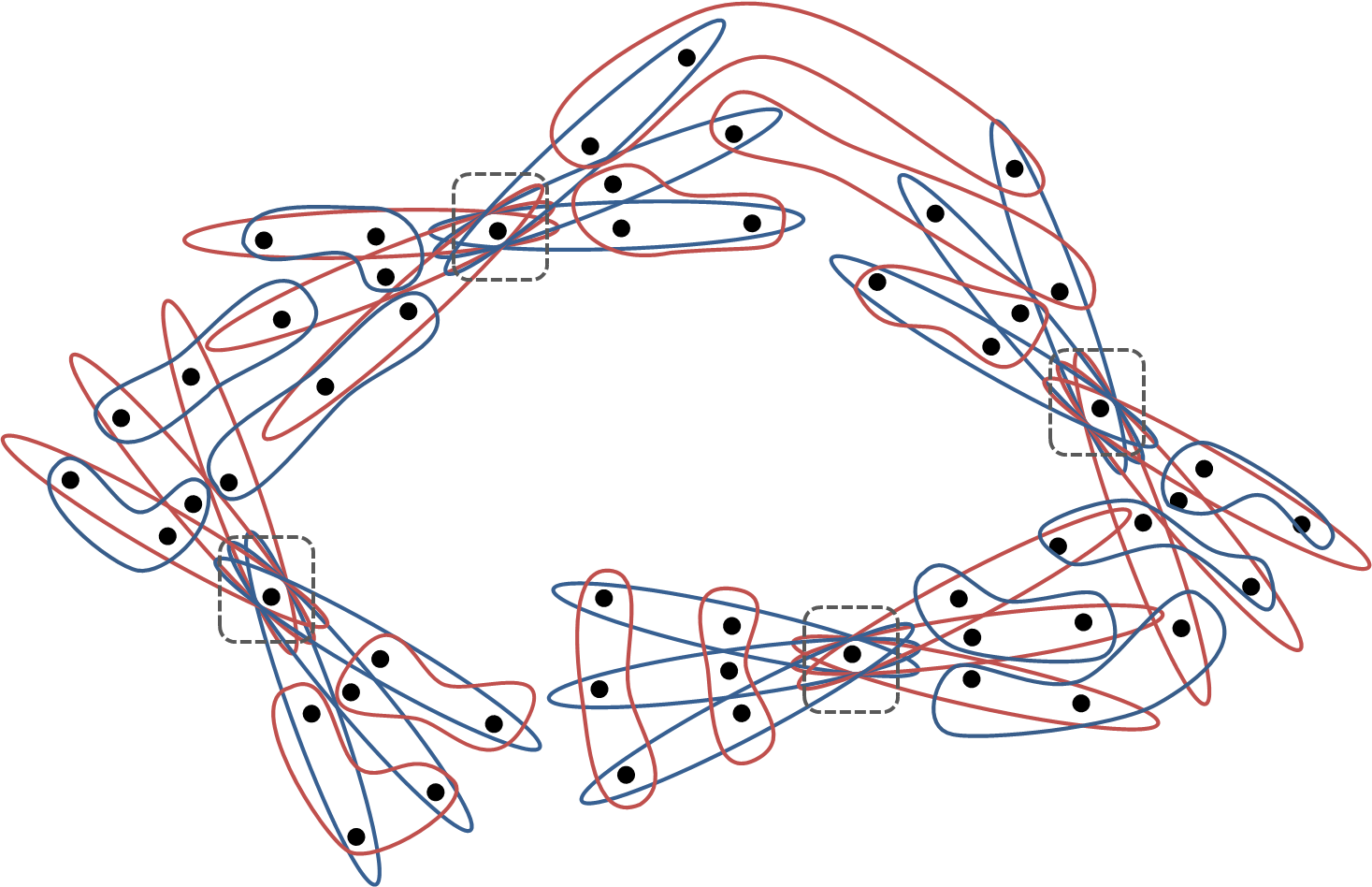}
	\caption{Primitive partitioned-core monomial sunflower $\mathcal H_3$, with an appropriate bicoloring.}
    \label{fig:mulipleCores}
\end{figure}
\end{eg}

We relax the disjointness property of the sunflowers above, and allow the sunflowers to intersect on core vertices.
Let $H$ be a hypergraph consisting of a collection of sunflowers $S_1, S_2, \ldots, S_{\ell}$, which may only intersect at their cores (known as a relaxed core sunflower), and a perfect matching on the non-core vertices of the sunflowers.
We will call such an $H$ a \emph{matched-petal relaxed-core sunflower}.

In the partitioned core case, the degree of each core vertex depends on exactly one sunflower. In contrast, allowing intersections in the cores causes the degree of a core vertex to depend on several sunflowers. This prompts the following notation.

\begin{notation}\label{notation:sunflowersSubset2}
For a matched-petal relaxed-core sunflower, define $\mathcal{I}(v)$ to consist of the set of indices $i$ such that vertex $v$ is in the core of sunflower $S_i$. In symbols $\mathcal{I}(v) := \{i\in\mathcal{I}(v):v\in S_i\}.$
\end{notation}
In addition, notation~\ref{notation:sunflowersSubset} extends to $H$.

\begin{theorem}[Relaxed Core Sunflower]  \label{thm:sunflowersRelaxedCore}
Let $H$ be a matched-petal relaxed-core sunflower, and adopt notation~\ref{notation:sunflowersSubset} and \ref{notation:sunflowersSubset2}.

The hypergraph $H$ is the support of a monomial hypergraph $\mathcal H$ if and only if there exists a partition $(G_\mathcal J, G_\mathcal K)$ of the connected components of $H-C$ such that for each vertex $v$ in the core of a sunflower:
\[\sum_{i\in \mathcal{I}(v)}\sum_{j\in \mathcal{J}(S_i)} m_j|G_{j}|=\sum_{i\in \mathcal{I}(v)}\sum_{k\in \mathcal{K}(S_i)} m_k|G_{k}|, \quad \text{where $m_j$, $m_k$ are integers.}\]

In addition, $\mathcal H$ is primitive if and only if the partition $(G_\mathcal J, G_\mathcal K)$ also satisfies the following two conditions:
\begin{enumerate}[i.]
\item no two subsets $\mathcal J'\subsetneq \mathcal J$ and $\mathcal K'\subsetneq \mathcal K$ satisfy for each vertex $v$ in the core of the sunflower,
\[\sum_{i\in \mathcal{I}(v)}\sum_{j\in \mathcal{J}(S_i)\cap \mathcal J'} m'_j|G_{j}|=\sum_{s\in \mathcal{I}(v)}\sum_{k\in \mathcal{K}(S_i)\cap \mathcal K'} m'_k|G_{k}|,\]
for some integers $m'_j\leq m_j$, $m'_k\leq m_k$, and
\item there are no integers $m'_j\leq m_j$, $m'_k\leq m_k\in\mathbb Z$, where the inequality is strict for at least one integer, such that the partition $(G_\mathcal J, G_\mathcal K)$ of the connected components of $H-C$ satisfies for each vertex $v$ in the core of a sunflower,
    \[\sum_{i\in \mathcal{I}(v)}\sum_{j\in \mathcal{J}(S_i)} m'_j|G_{j}|=\sum_{i\in \mathcal{I}(v)}\sum_{k\in \mathcal{K}(S_i)} m'_k|G_{k}|.\]
\end{enumerate}
\end{theorem}
\begin{proof}
Follow the proof of Proposition~\ref{prop:partitionedCoreSunflowers} with one important exception. The degree of a core vertex may now depend on several sunflowers, and so the quantity necessary to calculate the degree of a core vertex $v$ in part $\mathcal J$, for example, is $\sum_{i\in \mathcal{I}(v)}\sum_{j\in \mathcal{J}(S_i)} m'_j|G_{j}|$.
\end{proof}

Proposition~\ref{prop:partitionedCoreSunflowers} can be obtained as a corollary of Theorem~\ref{thm:sunflowersRelaxedCore}.

\begin{eg}
Figure~\ref{fig:RelaxedCore} demonstrates a simple primitive relaxed core monomial sunflower $\mathcal H_4$ with four sunflowers. Each multiplicity equals one.
\begin{figure}[hbt]
	\centering
	\includegraphics[scale=0.55]{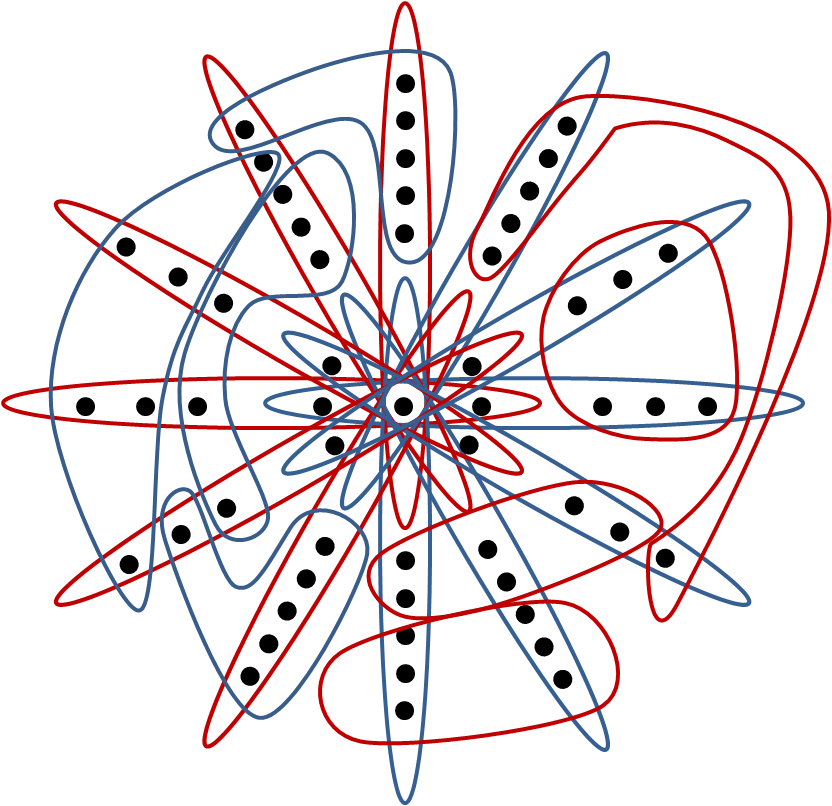}
	\caption{Primitive relaxed core monomial sunflower $\mathcal H_4$, with an appropriate bicoloring.}
    \label{fig:RelaxedCore}
\end{figure}
\end{eg}


\section{Bouquet Complexity}
\label{sec:ctrexamples}

The types of Graver elements studied in Propositions~\ref{prop:matchings}, \ref{prop:sunflowers} and~\ref{prop:partitionedCoreSunflowers} are, in fact, the \emph{simplest} ones that may occur beyond the graphs case.  They are the   natural direct generalization of the three types of monomial walks on graphs from  Theorem~\ref{thm:OH:characterization} (see also Figure~\ref{fig:3walksOnGraphs}), each of which is an example of a sparse bouquet. 
By allowing the edges to contain more than two vertices, we inadvertently and significantly increase the complexity of the primitive walks. This already occurs on $3$-uniform hypergraphs and on sparse bouquets.  We summarize some of the different ways in which the complexity has increased through the following examples.

In the case of graphs, it is known  (see \cite[Proposition~8.1.6]{Vill}) that the largest degree of a vertex in the support of a primitive walk is $4$, and that the edges in the support of the walk can be traversed at most twice.  
None of these restrictions extend to $d$-uniform hypergraphs ($d>2$), where a primitive walk can contain vertices with arbitrarily large degree; can contain arbitrarily many vertices of large degree; and a supporting edge can be used arbitrarily many times in a primitive walk.

\begin{figure}[hbt]
	\centering
	\includegraphics[scale=0.5]{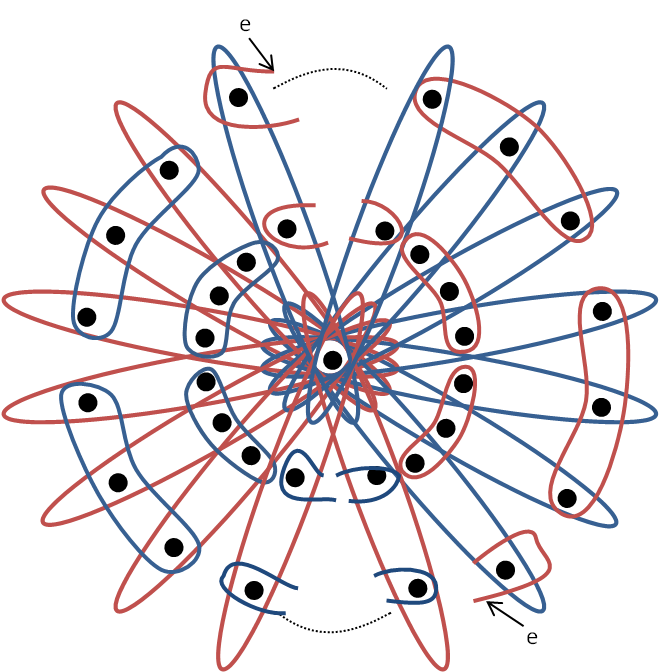}
	\caption{A primitive monomial hypergraph containing a vertex with arbitrarily large degree.}
	\label{fig:largeDegreeVertex}
\end{figure}

It is not hard to see that there exists exactly one primitive $d$-uniform monomial sunflower with core of size $d-1$: the one containing exactly $2d+2$ hyperedges. 
On the other hand, figure~\ref{fig:largeDegreeVertex} demonstrates a family of primitive 3-uniform bouquets on $n$ vertices with number of edges of the order of $n$ and containing a vertex with degree of the order of $n$.
\begin{figure}[hbt]
	\centering
	 \includegraphics[scale=0.5]{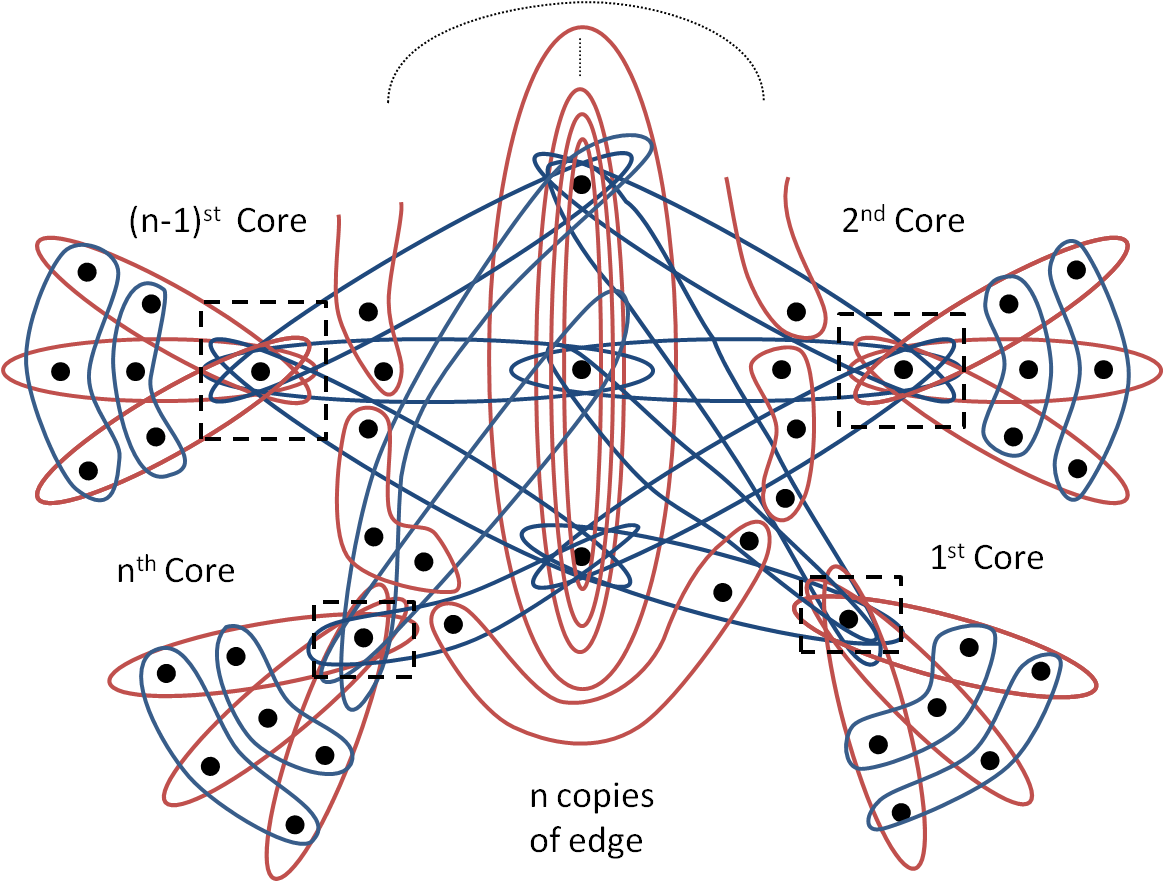}
	\caption{A primitive monomial hypergraph with arbitrarily many cores.}
	\label{fig:nthDegreeEdge}
\end{figure}
Figure~\ref{fig:nthDegreeEdge} demonstrates a family of $3$-uniform bouquets 
in which the number of cores, vertex degrees, and edge multiplicities can be arbitrarily high.

\section{Conclusion and open problems}\label{sec:problems}

Using the sparse hypergraphs as  a starting point, one should consider the problem of classifying supports of more general monomial hypergraphs. However, this problem is highly nontrivial. For example, Figure~\ref{fig:simpleNonPrimitiveBouquet} represents a simplest example of a non primitive monomial walk that is not of the form covered in Section~\ref{sec:sparseMonomialHypergraphs}. The hypergraph consists of two overlapped walks.
\begin{figure}[hbt]
    	\centering
    	 \includegraphics[scale=0.55]{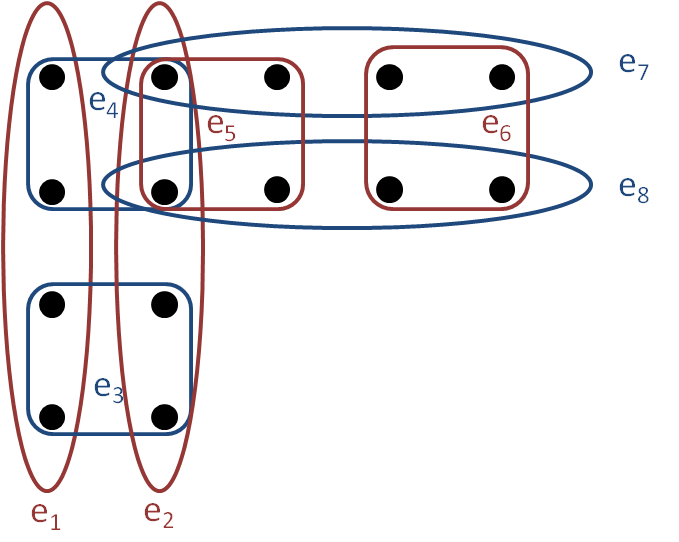}
    	\caption{A non-primitive monomial bouquet}
        \label{fig:simpleNonPrimitiveBouquet}
    \end{figure}
It is tempting to mimic the constructions for sparse bouquets and express this hypergraph as a union of, say, four sunflowers $S_1,\dots,S_4$ with two petals each:
$S_1=\{e_1,e_3\}$, $S_2=\{e_2,e_4\}$, $S_3=\{e_5,e_7\}$, $S_4=\{e_6,e_8\}$. Then, removing the cores from these sunflowers and counting the petal remnants in the connected components, with multiplicity $1$, as suggested by a direct generalization of Theorem~\ref{thm:sunflowersRelaxedCore}, would reveal that $S_3\cup S_4$  is a balanced set of edges. In this way,  one could detect that the monomial hypergraph $S_1\cup\dots\cup S_4$ is not primitive.

In general, however,  an arbitrary monomial hypergraph consists of sunflowers that overlap in unrestricted ways, and thus it is not guaranteed that removing the cores will disconnect and detect a subwalk.  In particular, a direct generalization of the results from Section~\ref{sec:sparseMonomialHypergraphs} does not hold without further assumptions.
\begin{problem}
	Find a criterion to detect primitivity of a given monomial hypergraph $H$ based on a subset of the edges in $H$.
\end{problem}
At the moment, the best result in this direction is Theorem~\ref{thm:balancedEdgesEquivalence}. For ``sparse" collections of sunflowers and matchings, Section~\ref{sec:sparseMonomialHypergraphs} offers more efficient criteria. But if  the hypergraph is more dense, an arbitrary sunflower decomposition alone does not seem to suffice.

More generally, it would be interesting, although difficult, to solve the following specific instance of the general problem of characterizing primitive walks.
\begin{problem}\label{bigProblem}
	Characterize (combinatorially) the supports of primitive monomial hypergraphs over an arbitrary $3$-uniform hypergraph.
\end{problem}
We remind the reader that  even for $3$-uniform hypergraphs, any of the types of bouquets that we describe can  appear in the Graver basis, as demonstrated throughout the paper. Thus Problem~\ref{bigProblem} is highly non-trivial; however, we expect it to have a nice answer for a subclass of $3$-uniform hypergraphs, for example, where it is possible to partition a vertex set in a way that restricts the types of bouquets that may appear. 
On the other hand, recall the ultimate goal of interest to statistical applications: we are looking for characterizations of the structure of monomial walks for a family of hypergraphs specified by conditions that are significantly more restrictive than uniformity. The family will come from a statistical model with a prescribed edge size, or edge types. Such conditions will necessarily restrict the ways in which edges of $H$ can intersect, and thus may imply the sparsity of walks supported on $H$.

Furthermore, one can quickly derive the entire Graver basis for the following two classes of hypergraphs. For \emph{any} $2$-regular hypergraph, it consists only of pairs of matchings as in Proposition~\ref{prop:matchings}. In particular, the Graver basis of a connected 2-regular hypergraph contains exactly one element. Secondly, consider the $d$-uniform hypergraph $C_r^{d}=(V,E)$, produced by letting $V=\bigsqcup_{1 \leq i\leq r} V_i$, with each $V_i$ having size $d/2$ ($d$ even), and edge set $E=\{V_i\cup V_j: \text{ for }1 \leq i < j\leq r\}$. A (primitive) monomial walk on $C_r^{d}$ corresponds to a (primitive) even closed walk in the complete graph $K_r$, as each hyperedge in $C_r^{d}$ is in a one-to-one correspondence with an edge in $K_r$.

\medskip

Our next problem concerns generating primitive monomial hypergraphs. It is possible to generate such hypergraphs by contraction or ``gluing" operations using known monomial hypergraphs (and sparse bouquets in particular) as building blocks. For example, one may produce $d$-uniform monomial sunflowers from a $d$-uniform pair of perfect matchings by identifying selected sets of independent vertices. This operation preserves primitivity. The converse is only true if vertex-identification takes place within connected components. In general, one may also glue two hypergraphs $H_1$, $H_2$ along the vertices of an edge to produce a new hypergraph $H$, introducing new primitive elements in $I_H$ constructed from binomials from $I_{H_1}$ and $I_{H_2}$. The support of the monomial hypergraph in Figure~\ref{fig:primitiveBouquetGluing} was partially constructed from the support of four monomial sunflowers with this gluing operation. It is not hard to see that $\mathcal B$ is a monomial hypergraph.

\begin{figure}[hbt]
	\centering
	\includegraphics[scale=0.4]{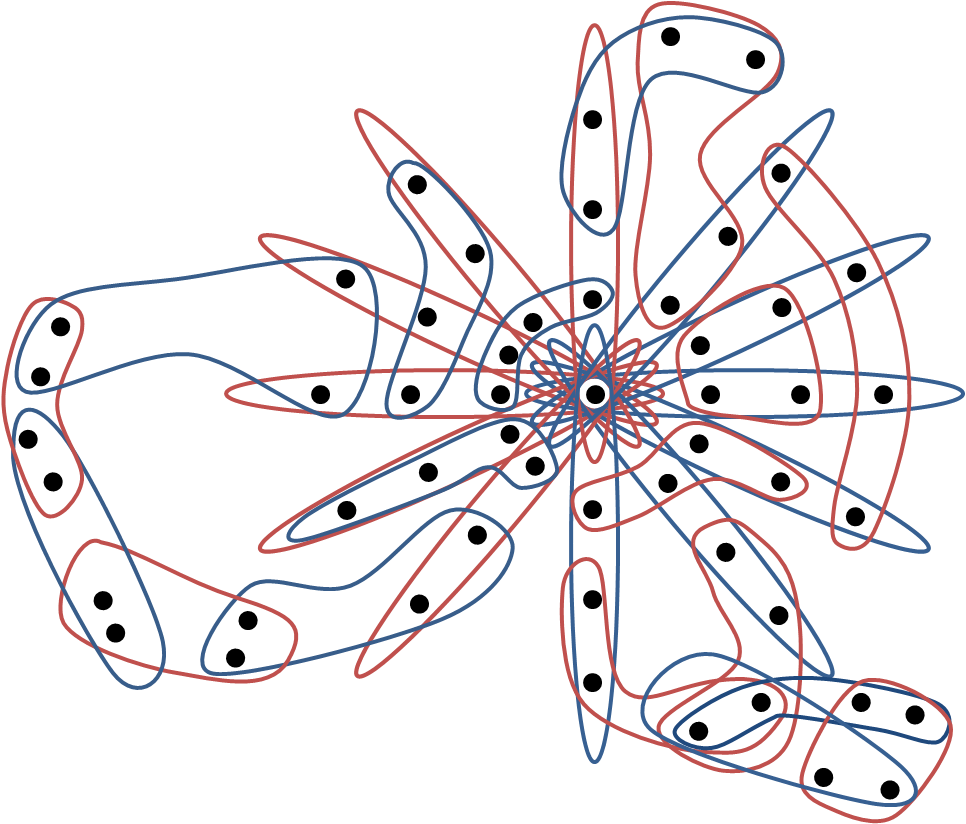}
	\caption{A primitive bouquet $\B$ partially produced by gluing four sparse bouquets.}
	\label{fig:primitiveBouquetGluing}
\end{figure}

A general gluing operation for graphs, based on a toric fiber product, is described in~\cite{EKS}. It preserves many properties of toric ideals.

\begin{problem}
Find a gluing operation on hypergraphs that preserves primitivity of monomial walks. Is it possible to generate all primitive monomial hypergraphs starting from sparse bouquets?
\end{problem}

\medskip

The primitive walks in hypergraphs are clearly much more general than in the case of graphs, as Section~\ref{sec:ctrexamples} shows.
Nevertheless,  we expect many of the other properties of the coordinate ring of $I_H$ to have combinatorial interpretations. For example, the first author studies the degree bounds on the minimal generators of $I_H$ in \cite{GrPe}, generalizing the theorems known for graphs.

Finally, the problems of finding nice term orders and reduced Gr\"obner bases of $I_H$, characterizing Cohen-Maculayness and normality of the coordinate ring and the corresponding polytope, and relating known coloring-inspired properties of hypergraphs to various invariants of the toric ideal $I_H$ are wide open.



\section*{Acknowledgements}    The authors thank Gy\"orgy Tur\'an for inspiring discussions, at the start of this project, about the fundamental problem of walks in hypergraphs; and also Elizabeth Gross for suggesting the name ``bouquet", and proofreading an early version of the manuscript. The second author is grateful to Amitava Bhattacharya for introducing her to combinatorial discrepancy.
Finally, we are grateful to the anonymous referee for thoughtful and thorough comments that helped improve this work.

The authors acknowledge support by grant FA9550-12-1-0392 from the U.S. Air Force Office of Scientific Research (AFOSR) and the Defense Advanced Research Projects Agency (DARPA). 
The AMS Simons Travel Grant has helped the first author tremendously during the first year of the project. 

\bibliography{toric_algebra_hypergraphs}
\bibliographystyle{amsalpha}

\end{document}